\documentclass[english,russian,12pt]{amsart} 
\usepackage{nicefrac}
\usepackage[inline]{enumitem}
\usepackage{enumitem}

\usepackage{amssymb}
\usepackage[usenames,dvipsnames,svgnames,table]{xcolor}
\definecolor{myRed}{rgb}{0.9,0.,.2}
\definecolor{myBlue}{rgb}{0.,0.,.6}
\definecolor{myGreen}{rgb}{0.1,0.7,0.1}
\definecolor{myViolet}{rgb}{102,0,153}
\usepackage[colorlinks=true,urlcolor=myBlue,linkcolor=myBlue,citecolor=myGreen,pagebackref=true]{hyperref}

\usepackage{fouriernc}
\usepackage[scaled=0.775]{helvet}
\usepackage{courier}
\usepackage[marginparwidth=2cm]{geometry}
\usepackage{array,float}

\usepackage{todonotes}
\usepackage{tikz, tikz-3dplot}
\usepackage{tikzscale}
\usetikzlibrary{shapes,arrows}
\usetikzlibrary{fit,positioning}
\usetikzlibrary{patterns,decorations.pathreplacing}
\usepackage{xkeyval}
\usepackage{moreverb}
\usepackage{epic}
\usepgfmodule{shapes,plot,decorations}

\usepackage{booktabs} 

\usepackage[normalem]{ulem} 

\newcommand{\R}{\mathbb{R}}

\newcommand{\HH}{\mathbb{H}}
\newcommand{\EE}{\mathbb{E}}

\usepackage[warn]{mathtext}
\newcommand\El{\Lambda}

\newcommand{\Vol}{\mathrm{Vol}}

\usepackage{mathrsfs}

\theoremstyle{plain}
\newtheorem{theorem}{\indent Theorem}[section]

\newtheorem{proposition}[theorem]{\indent Proposition}
\newtheorem{corollary}[theorem]{\indent Corollary}
\newtheorem{lemma}[theorem]{\indent Lemma}

\renewenvironment{proof}
	{\par\indent{\bf Proof.}} 
	{\hfill$\scriptstyle\blacksquare$}

\makeatletter
  \def\section{\@startsection{section}{2}%
    {\z@}{.5\linespacing\@plus.7\linespacing}{.5em}%
    {\normalfont\bfseries\centering}}
\def\@secnumfont{\bfseries}
\makeatother

\theoremstyle{definition}
\newtheorem{example}[theorem]{Example}

\theoremstyle{remark}
\newtheorem{remark}[theorem]{Remark}

\usepackage{ulem}

\title{On volumes of hyperbolic right-angled polyhedra}

\author{Stepan Alexandrov}
\address{Dept. of Discrete Mathematics, Moscow Institute of Physics and Technology, Russia}
\email{aleksandrov.sa@phystech.edu}

\author{Nikolay Bogachev}
\address{The Kharkevich Institute for Information Transmission Problems} 
\address{Dept. of Discrete Mathematics, Moscow Institute of Physics and Technology, Russia}
\email{nvbogach@mail.ru}

\author{Andrei Egorov}
\address{Novosibirsk State University, Novosibirsk, Russia}
\email{a.egorov2@g.nsu.ru}

\author{Andrei Vesnin}
\address{Tomsk State University, Tomsk, Russia}
\address{Higher School of Economics, Moscow, Russia}
\email{vesnin@math.ncs.ru}

\begin{document}

\setcounter{tocdepth}{1}

\begin{abstract}
In this paper we obtain new upper bounds on volumes of right-angled polyhedra in hyperbolic space $\HH^3$ in three different cases: for ideal polyhedra with all vertices on the ideal hyperbolic boundary, for compact polytopes with only finite (or usual) vertices, and for finite volume polyhedra with vertices of both types.
\end{abstract}



\maketitle

\section{Introduction}

Studying volumes of hyperbolic polyhedra and hyperbolic 3-manifolds is a fundamental problem in geometry and topology. We discuss  finite volume polyhedra in hyperbolic $n$-space $\HH^n$ with all dihedral angles equal to $\frac{\pi}{2}$, and call them \emph{right-angled} polyhedra. It is known that there are no compact right-angled polytopes in $\HH^n$ for dimension $n > 4$, and there are no finite volume right-angled polyhedra in $\HH^n$ for $n > 12$~\cite{PV05, Dufour}, with examples known only for $n < 9$ \cite{PV05}. It is worth mentioning that recently studying volumes of ideal right-angled polyhedra has become more attractive and important in view of the problem of maximal~\cite{Belletti} and minimal (see, e.g.~\cite{Kol12}) volumes, and the conjecture about hyperbolic right-angled knots~\cite{CKP19}.

In the present paper, we consider right-angled hyperbolic polyhedra in $\HH^3$.  An initial list of ideal right-angled polyhedra is given in~\cite{EV20}, and of compact ones in~\cite{Inoue}. A detailed discussion of constructions of hyperbolic $3$-manifolds from right-angled polyhedra can be found in the recent survey~\cite{V17}. In the compact case, there is a number of constructions related to small covers, see e.g. ~\cite{DJ91, BP16}. Also, right-angled polytopes are useful for constructing hyperbolic 3-manifolds that bound geometrically~\cite{KMT}. Let us also mention several works on the interplay between the arithmetic hyperbolic reflection groups and arithmeticity of hyperbolic links~\cite{Kel, MMT20}. Here it turns out to be very useful that fundamental groups of some hyperbolic link complements are commensurable with hyperbolic reflection groups since Vinberg's theory of reflection groups~\cite{Vin85} can be applied.

In 1970, Andreev \cite{Andreev1, Andreev2} (see also \cite{Roeder_after_Andreev}) obtained his famous characterization of hyperbolic acute--angled $3$-polyhedra of finite volume. For right--angled polyhedra, Andreev's theorems provide simple necessary and sufficient conditions for realizing a given combinatorial type as a compact, finite--volume or ideal polyhedron in $\HH^3$. Such realizations are determined uniquely up to isometry. Thus one can expect that  geometric invariants of these polyhedra can be estimated via combinatorics. Lower and upper bounds for volumes of right-angled hyperbolic polytopes using the number of vertices were obtained by Atkinson~\cite{Atkinson}. 

In this paper we obtain new upper bound on volumes of ideal right--angled polyhedra (see Theorem~\ref{theorem:main-ideal}),  compact right--angled polytopes (see Theorem~\ref{theorem:main-compact}) and finite--volume right--angled polyhedra with both finite and ideal vertices (see Theorem~\ref{theorem:main-compact-ideal}). 

Recall \cite{Vinberg} that volumes of hyperbolic $3$-polyhedra can usually be expressed via the \emph{Loba\-chevsky function}  $$
\El (x) = - \int_{0}^{x} \log | 2 \sin t | dt. 
$$
In order to formulate the main results more conveniently, we define two constants depending on Lobachevsky function's values at certain points. The first one, $v_{8} = 8 \El (\pi/4)$, equals the volume of the regular ideal hyperbolic octahedron. Up to six decimal places $v_{8}  = 3.663862$. The second one, $v_3 = 3 \El (\pi/3)$, equals the volume of the regular ideal hyperbolic tetrahedron. Up to six decimal places $v_3  = 1.014941$.

\subsection{Ideal right-angled hyperbolic $3$-polyhedra} 

Recall that if $P \subset \mathbb H^3$ is an ideal right-angled polyhedron with $V$ vertices, then $V \geq 6$. Moreover, $V = 6$ if and only if $P$ is an octahedron, which can be described as the antiprism $A(3)$  with triangular bases. Thus, $\Vol(A(3)) = v_8$. 

The volume formula for antiprisms $A(n)$ (i.e. ideal right-angled polytopes with $V = 2n$ vertices, two $n$-gonal bases and $2n$ lateral triangles), $n \geq 3$, was obtained by Thurston~\cite[Chapter 6 \& 7]{Thurston}:
$$
\Vol (A(n)) = 2n \left[ \El \left(\frac{\pi}{4} + \frac{\pi}{2n} \right) + \El \left( \frac{\pi}{4} - \frac{\pi}{2n} \right) \right]. 
$$
For example, up to six decimal places, $\Vol(A(4)) = 6,023046$.

\medskip
In 2009, Atkinson obtained~\cite[Theorem~2.2]{Atkinson} the following upper and lower bounds for volumes via the number of vertices. Let $P$ be an ideal right-angled hyperbolic $3$-polytope with $V \geq 6$ vertices, then  
$$\frac{v_8}{4} \cdot V - \frac{v_8}{2}   \leqslant \Vol (P) \leqslant  \frac{v_{8}}{2} \cdot V -2 v_{8}.$$

Both inequalities are sharp when $P$ is a regular ideal octahedron (i.e. for $V=6$). Moreover, they are asymptotically sharp in the following sense: there exists a sequence of ideal right-angled polytopes $P_{i}$ with $V_{i}$ vertices such that $\operatorname{Vol} (P_{i}) / V_{i} \to \frac{v_{8}}{2}$ as $i \to +\infty$.  

There are no ideal right-angled polytopes with $V=7$, and $V=8$ if and only if $P$ is the antiprism $A(4)$ with quadrilateral  bases. The following upper bound was obtained in~\cite[Theorem~2.2]{EV20_1}. Let $P$ be an ideal right-angled hyperbolic $3$-polyhedron with $V \geq 9$ vertices. Then 
$\Vol (P) \leqslant  \frac{v_{8}}{2} \cdot V -  \frac{5v_{8}}{2}.$

The inequality is sharp when $P$ is the double of a regular ideal octahedron along a face (i.e. for $V=9$). The graphs of the above lower and upper bounds in comparison to the volumes of ideal right-angled polyhedra up to $21$ vertices can be found in~\cite{EV20_1}.  

\begin{theorem}\label{theorem:main-ideal} 
Let $P$ be an ideal right-angled hyperbolic $3$-polyhedron with $V$ vertices. Then the  following conditions hold.
\begin{itemize}
\item[(1)] If $V > 24$, then $\Vol(P) \leqslant \frac{v_8}{2} \cdot V - 3 v_8.$ 
\item[(2)] If $P$ has a $k$-gonal face, $k \geq 3$, then $\Vol(P) \leqslant \frac{v_8}{2} \cdot V  - \frac{k+5}{4} v_8.$
\item[(3)] If $P$ has only triangular and quadrilateral faces with $V \ge 73$, then
$$\Vol(P) \leqslant  \frac{v_8}{2} \cdot V - \left( 9 v_8 - 20 v_3 \right).$$
\end{itemize} 
\end{theorem}

The proof of Theorem~\ref{theorem:main-ideal} is given in Section~\ref{sec3}.

\subsection{Compact right-angled hyperbolic $3$-polytopes}

It is well-known that for a compact right-angled polytope in $\HH^3$ with $V$ vertices we have either $V = 20$ or $V \ge 24$. Moreover, $V = 20$ if and only if $P$ is a regular right-angled dodecahedron. 

The volume formula is known for an infinite series of compact right-angled \emph{Loebell polytopes} $L(n)$, $n \geq 5$, with $V = 4n$, two $n$-gonal bases and $2n$ lateral pentagonal faces. In particular, $L(5)$ is a regular dodecahedron. 
By~\cite{V98}, for $n \geq 5$, the volume of $L(n)$ is
$$
\Vol(L(n)) = \frac{n}{2} \left[ 2 \El (\theta) + \El \left( \theta + \frac{\pi}{n} \right) + \El \left( \theta - \frac{\pi}{n} \right) - \El \left( 2 \theta - \frac{\pi}{2} \right) \right], 
$$
where $\theta = \frac{\pi}{2} - \arccos \left(\frac{1}{2 \cos (\pi/n)}\right)$. 

\medskip 
Two-sided bounds for volumes of compact right-angled hyperbolic $3$-polytopes were obtained by Atkinson \cite[Theorem~2.3]{Atkinson}.  Namely, if $P$ is a compact right-angled hyperbolic $3$-polytope with $V$ vertices, then  
$$
\frac{v_8}{32} \cdot V - \frac{v_8}{4} \leq \Vol(P) <  \frac{5v_3}{8} \cdot V  - \frac{25}{4} v_3. 
$$
There exists a sequence of compact right-angled polytopes $P_i$ with $V_i$ vertices such that 
$\Vol(P_i)/V_i \to \frac{5v_3}{8}$ with $i \to +\infty$. 

The upper bound can be improved if we exclude the case of the dodecahedron. Indeed, by~\cite[Theorem~2.4]{EV20_1}, if $P$ is a compact right-angled hyperbolic $3$-polytope with $V \geq 24$ vertices, then 
$\Vol(P) \leq  \frac{5v_3}{8} \cdot V - \frac{35}{4} v_3.$

\begin{theorem} \label{theorem:main-compact} 
Let $P$ be a compact right-angled hyperbolic polytope with $V$ vertices. Then the following inequalities hold.
\begin{itemize}
\item[(1)] If $V > 80$, then  
$\Vol(P) \leq  \frac{5v_3}{8} \cdot V -  10 v_3.$
\item[(2)] If $P$ has a $k$-gonal face, $k \geq 5$, then  
$$
\Vol(P) \leq    \frac{5v_3}{8} \cdot V  - \frac{5k + 35}{8} v_3.
$$
\end{itemize} 
\end{theorem}

The proof of Theorem~\ref{theorem:main-compact} is given in Section~\ref{sec4}.

\subsection{Right-angled hyperbolic $3$-polyhedra with finite and ideal vertices}
As well as in the ideal and compact cases, Atkinson obtained volume bounds for right-angled hyperbolic polyhedra having both finite and ideal vertices \cite[Theorem~2.4]{Atkinson}. If $P$ is a finite-volume right-angled hyperbolic polyhedron with $V_{\infty}$ ideal vertices and $V_F$ finite vertices, then   
\begin{equation}
\frac{v_8}{8} \cdot V_{\infty} + \frac{v_8}{32} \cdot V_F -  \frac{v_8}{4} \leq \Vol(P)< 
\frac{v_8}{2} \cdot V_{\infty} + \frac{5 v_3}{8} \cdot V_F -  \frac{v_8}{2}.   \label{eqn1}
\end{equation}
	
\medskip 	

Provided more combinatorial information about $P$, we are able to improve the latter bound as follows. 

\begin{theorem} \label{theorem:main-compact-ideal} 
	Let $P$ be a finite-volume right-angled hyperbolic $3$-polyhedron with $V_\infty$ ideal vertices and $V_F$ finite vertices. If $V_\infty + V_F > 17$, then 
$$ 
\Vol(P)<  \frac{v_8}{2} \cdot V_{\infty} + \frac{5 v_3}{8} \cdot V_F - \left( v_8 + \frac{5v_3}{2} \right).  
$$
\end{theorem}

The proof of Theorem~\ref{theorem:main-compact-ideal} is given in Section~\ref{sec5}.

\subsection{Acknowledgments} 
This research project was started during the Summer 2021 Research Program for Undergraduates organized by the Laboratory of Combinatorial and Geometric Structures at MIPT. 
The authors were supported by the Theoretical Physics and Mathematics Advancement Foundation ``BASIS''. A.V. was also supported in part by the RSF grant no. 20-61-46005. 

\section{Combinatorics of hyperbolic right-angled polyhedra}

\subsection{Hyperbolic polyhedra}

Let $\R^{d,1}$ be the vector space $\R^{d+1}$ equipped with a scalar  product $\langle \cdot, \cdot \rangle$ of signature $(d,1)$, and let $f_d$ be the associated quadratic form. The coordinate representation of $f_d$ with respect to an appropriate basis of $\R^{d,1}$ is
$$ f_d(x) =  -x_0^2 + x_1^2 + \dotsm + x_d^2.$$
The \emph{hyperbolic $d$-space} $\HH^{d}$ is the upper half-sheet (a connected component) of the hyperboloid $f_d(x) = -1$: 
$$
\HH^{d} = 
\{ x \in  \R^{d,1} \mid f_d (x) = -1\ \textrm{ and }\, x_{0} > 0 \}.
$$
In this model, points from the ideal hyperbolic boundary correspond to isotropic vectors:
$$
\partial \HH^d = \{x \in  \R^{d,1} \mid f_d(x)=0 \textrm{ and }\, x_{0} > 0\}/\R^*.
$$

A \emph{convex hyperbolic $d$-polyhedron} is the intersection, with non-empty interior, of a finite family of closed half-spaces in hyperbolic $d$-space $\HH^d$ . A \emph{hyperbolic Coxeter $d$-polyhedron} is a convex hyperbolic $d$-polyhedron $P$ all of whose dihedral angles are integer sub-multiples of $\pi$, i.e. of the form $\frac{\pi}{m}$ for some integer $m \geqslant 2$. A hyperbolic Coxeter polyhedron is called \emph{right-angled} if all its dihedral angles are $\frac{\pi}{2}$. A generalized\footnote{A \textit{generalized convex polyhedron} $P$ is the intersection, with non-empty interior, of possibly infinitely many closed half-spaces in hyperbolic $d$-space such that every closed ball intersects only finitely many bounding hyperplanes of $P$} convex polyhedron is said to be \textit{acute-angled} if all its dihedral angles do not exceed $\frac{\pi}{2}$. 

It is known that generalized Coxeter polyhedra are the natural fundamental domains of discrete reflection groups in spaces of constant curvature, see~\cite{Vin85}.

A convex $d$-polyhedron has \emph{finite volume} if and only if it is the convex hull of finitely many points of the closure $\overline{\HH^d} = \HH^d \cup \partial \HH^d$. If a convex polyhedron is \emph{compact}, then it is called a \emph{polytope}, and is a convex hull of finitely many proper points of $\HH^d$, which are its \emph{proper} (or \emph{finite}) vertices. And, finally, a convex polyhedron is called \emph{ideal}, if all its vertices are situated on the ideal hyperbolic boundary $\partial \HH^d$ (such vertices are also called \emph{ideal}). It is known that compact acute-angled polytopes (in particular, compact Coxeter polytopes) in $\HH^d$ are \emph{simple}, i.e. every vertex belongs to exactly $d$ facets (and $d$ edges).

Two compact polytopes $P$ and $P'$ in Euclidean space $\EE^d$ are \emph{combinatorially equivalent} if there is a bijection between their faces that preserves the inclusion relation. A combinatorial equivalence class is called a \emph{combinatorial polytope}. Note that if a hyperbolic polyhedron $P \subset \HH^d$ is of finite volume, then the closure $\overline{P}$ of $P$ in $\overline{\HH^d}$ is combinatorially equivalent to a compact polytope of $\EE^d$. 

\medskip

The following theorem is a special case of Andreev's theorem (see \cite{Andreev1, Andreev2}).

\begin{theorem}\label{theorem:Andreev} 
Let $\mathcal{P}$ be a combinatorial $3$-polytope. There exists a finite--volume  right-angled hyperbolic $3$-polyhedron $P \subset \overline{\HH^3}$ that realizes $\mathcal{P}$ if and only if:
\begin{enumerate}
    \item $\mathcal{P}$ is neither a tetrahedron, nor a triangular prism;
    \item every vertex of $P$ belongs to at most four faces;
    \item if $f$, $f'$, and $f''$ are faces of $\mathcal{P}$, and $e' = f \cap f'$, $e'' = f \cap f''$ are non-intersecting edges, then $f'$ and $f''$ do not intersect each other;
    \item there are no faces $f_1$, $f_2$, $f_3$, $f_4$ such that $e_i := f_i \cap f_{i + 1}$ (indices $\mathrm{mod}\, 4$) are pairwise non-intersecting edges of $\mathcal{P}$.
\end{enumerate}
\end{theorem}

In a right-angled polyhedron $P \subset \overline{\HH^3}$, a vertex $v$ lies in $\HH^3$ (i.e. is finite) if and only if it belongs to exactly three faces of $P$. If a vertex $v$ is contained in four faces of $P$, then $v \in \partial\HH^d$.

\subsection{Combinatorics of ideal right-angled hyperbolic $3$-polyhedra}

Let $P$ be an ideal hyperbolic right-angled $3$-polyhedron. Let $V$ be the number of vertices, $E$ the number of edges, and $F$ the number of faces of $P$. The Euler characteristic of $P$ equals $V - E + F = 2$. Every ideal vertex of a finite volume right-angled hyperbolic $3$-polyhedron is contained in exactly four edges which implies $4V = 2E$, and hence $F = V + 2$. Let $p_k$ denote the number of $k$-gonal faces of $P$. Then the previous equalities provide
$\sum_{k \geqslant 3} p_k = F$, $\sum_{k \geqslant 3} k p_k = 4V$, and $p_3 = 8 + \sum_{k \geqslant 5} (k - 4) p_k.$

\medskip

We say that two vertices of $P$ are \emph{adjacent} if they are connected by an edge. Two vertices are \emph{quasi-adjacent} if they belong to the same face but are not adjacent.

\begin{example}
    \begin{figure}[ht]
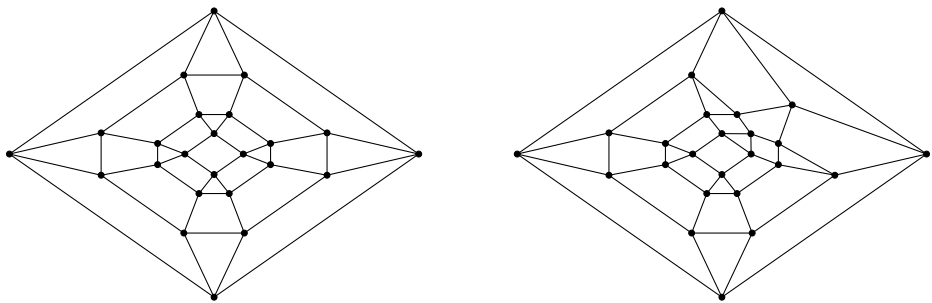

        \centering
        
        \includegraphics[width=\textwidth/3]{ideal_with_24_vert_1.tikz}
        \hspace{1cm}
        \includegraphics[width=\textwidth/3]{ideal_with_24_vert_2.tikz}

        \caption{Ideal right-angled polyhedra with $N = 24$ vertices.}
        \label{figure:ideal-24}
    \end{figure}
    
    Two ideal right-angled polyhedra with $24$ vertices from~\cite{EV20} are shown in Figure~\ref{figure:ideal-24}. Each vertex of those polyhedra belongs to exactly one triangular face and three quadrilateral faces. Thus, each vertex of those polyhedra has exactly three quasi-adjacent vertices. 
\end{example}

\begin{proposition} \label{proposition:ideal-vertex-with-4-qn}
    Let $P$ be an ideal right-angled hyperbolic $3$-polyhedron with $V > 24$ vertices. Then there is a vertex that has at least $4$ quasi-adjacent ones. 
\end{proposition}

\begin{proof}
Let $q(v)$ be the number of vertices that are quasi-adjacent to $v$. Then the average number of quasi-adjacent vertices in $P$ equals
    \begin{multline*}
        \frac{\sum_v q(v)}{V} =
        \frac{1}{V} \sum_{k \geqslant 3} k (k - 3) p_k = \frac{1}{V} \sum_{k \geqslant 3} k^2 p_k - \frac{3}{V} \sum_{k \geqslant 3} k p_k = \\
         = \frac{1}{V} \sum_{k \geqslant 3} k^2 p_k - 12 = \frac{1}{V} \Big[3^2\,p_3 + 4^2\,\Big(F - p_3 - \sum_{k \geqslant 5} p_k\Big) + \sum_{k \geqslant 5} k^2 p_k\Big] - 12,
         \end{multline*}
        which is equal to 
         \begin{multline*}
        \frac{1}{V} \Bigg[3^2\,\Big(8 + \sum_{k \geqslant 5} (k - 4)p_k\Big) + 4^2\,\Big(F - p_3 - \sum_{k \geqslant 5} p_k\Big) + \sum_{k \geqslant 5} k^2 p_k\Bigg] - 12 = \\
        = \frac{1}{V} \Bigg[3^2\,\Big(8 + \sum_{k \geqslant 5} (k - 4)p_k\Big) + 4^2\,\Big(V - 6 - \sum_{k \geqslant 5} (k - 3) p_k\Big) + \sum_{k \geqslant 5} k^2 p_k\Bigg] - 12 = \\
        = 4 - \frac{24}{V} + \frac{1}{V} \sum_{k \geqslant 5} (k^2 - 7k + 12) p_k \geqslant 4 - \frac{24}{V} > 3.
    \end{multline*}
Therefore, there is a vertex in $P$ that has at least $4$ quasi-adjacent vertices.
\end{proof}

\begin{remark} \label{remark:ideal-k-face-qn}
    Each vertex of a $k$-gonal face has at least $k - 3$ quasi-adjacent ones. So if a polytope has a $k$-gonal face for $k \ge 7$ then it has a vertex with at least $4$ quasi-adjacent ones. 
\end{remark}

\medskip

We say that a face $f$ and a vertex $v$ are \emph{incident} if $v$ belongs to $f$. A face $f$ and a vertex $v$ are \emph{quasi-incident} if they are not incident, but $v$ has an incident face $f'$ such that $f'$ shares an edge with $f$.  

\begin{proposition} \label{proposition:ideal-isolated-vertex}
Let $P$ be an ideal right-angled hyperbolic polyhedron with $V > 72$ vertices and without $k$-gonal faces for all $k \geqslant 5$. Then there is a vertex without  incident and quasi-incident triangular faces.
\end{proposition}

\begin{proof}
Let $F$ denote the number of faces of $P$. Because there are no $k$-gonal faces with $k\geq 5$, the set of faces of $P$ contains $8$ triangles and $F-8$ quadrilaterals. 
\begin{figure}[ht]
        \centering 
        \includegraphics[width=\textwidth/4]{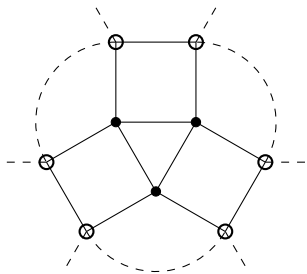}
        \caption{Incident and quasi-incident vertices of a $3$-gonal face.}
        \label{figure:qn-3-gonal-face}
    \end{figure}
Every triangular face is incident or quasi-incident to at most $9$ vertices (see Figure~\ref{figure:qn-3-gonal-face}). Therefore, at most $72$ vertices of $P$ can be incident or quasi-incident to triangular faces. 
\end{proof}

\subsection{Combinatorics of compact right-angled hyperbolic $3$-polytopes} \label{sec2}

Let $P$ be a compact right-angled hyperbolic $3$-polytope. Let $V$ denote the number of vertices, $E$ the number of edges, and $F$ the number of faces. The Euler characteristic of $P$ equals $V - E + F = 2$. Every vertex of a compact right-angled hyperbolic $3$-polytope is incident to three edges which implies $3V = 2E$ and $F = \frac{V}{2} + 2$. Let $p_k$ denote the number of $k$-gonal faces of $P$. By Theorem~\ref{theorem:Andreev} $p_3 = 0$ and $p_4 = 0$, and the previous equalities imply
$$
\sum_{k \geqslant 5} p_k = F, \qquad \sum_{k \geqslant 5} k p_k = 3V, \quad \text{and} \quad p_5 = 12 + \sum_{k \geqslant 7} (k - 6) p_k.$$

An edge $e$ and a vertex $v$ are \emph{incident} if $v$ is one of the two vertices that $e$ connects.  We say that an edge $e$ and a vertex $v$ are \emph{quasi-incident} if they are not incident, but at least one vertex of $e$ belongs to the same face as $v$.  

Since each vertex of a compact right-angled hyperbolic polyhedron $P$ is trivalent, we have four faces $f_1$, $f_2$, $f_3$, and $f_4$ arranged around each edge $e$ of $P$ as shown in Figure~\ref{figure:edge-and-adjacent-faces}. If $f_i$ is $k_i$-gonal, then number of vertices quasi-incident to $e$ is equal to $\sum_{i=1}^4 k_i - 10$. 

\begin{figure}[ht]
\centering
    \includegraphics{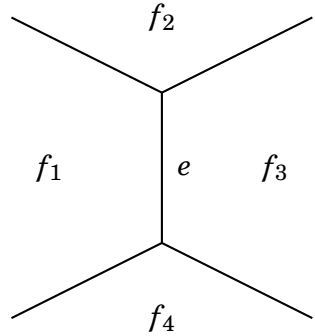}
    \caption{Edge $e$ and faces around.}
    \label{figure:edge-and-adjacent-faces}
\end{figure}

\begin{example}
An edge of fullerene C80 (see Figure~\ref{figure:C80}) has either $12$ or $13$ quasi-incident vertices. 
    \begin{figure}[h]
        \centering
        \includegraphics{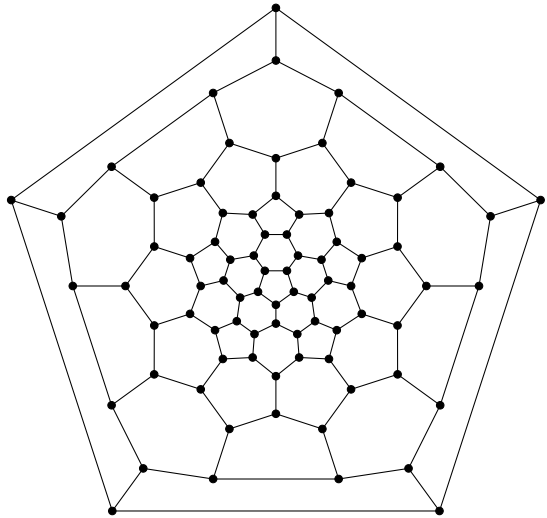}
        \caption{Fullerene C80.}
        \label{figure:C80}
    \end{figure}
Indeed, there are two types of edges. An edge of C80  either has two pentagons and two hexagons around itself, or has one pentagon and three hexagons. 
\end{example}

\begin{lemma} \label{lemma2.7}
    Let $P$ be a compact right-angled hyperbolic polytope with $V > 80$ vertices. Then there is an edge with at least $14$ quasi-incident vertices.
\end{lemma}

\begin{proof}
Let $q(e)$ be the number of vertices that are quasi-incident to an edge $e$. Then the average number of quasi-incident vertices in $P$ equals
\begin{multline*}
\frac{\sum_{e} q(e)}{E} = \frac{1}{E} \left[\sum_{k \geqslant 5} k (k - 2) p_k + \sum_{k \geqslant 5} k (k - 3) p_k\right] = \frac{2}{E} \sum_{k \geqslant 5} k^2 p_k - \frac{5}{E} \sum_{k \geqslant 5} k p_k = \\
        = \frac{2}{E} \sum_{k \geqslant 5} k^2 p_k - 10 = \frac{2}{E} \left[5^2 p_5 + 6^2 p_6 + \sum_{k \geqslant 7} k^2 p_k\right] - 10 = \\
        = \frac{2}{E}\left[5^2\,\left(12 + \sum_{k \geqslant 7} (k - 6) p_k\right) + 6^2\,\left(F - p_5 - \sum_{k \geqslant 7} p_k\right) + \sum_{k \geqslant 7} k^2 p_k\right] - 10 = \\
        = \frac{2}{E}\left[5^2\,\left(12 + \sum_{k \geqslant 7} (k - 6) p_k\right) + 6^2\,\left(\nicefrac{E}{3} - 10 - \sum_{k \geqslant 7} (k - 5) p_k\right) + \sum_{k \geqslant 7} k^2 p_k\right] - 10 = \\
        = 14 - \frac{120}{E} - \frac{2}{E} \sum_{k \geqslant 7} (k^2 - 11 k + 30) p_k \geqslant 14 - \frac{120}{E} = 14 - \frac{80}{V} > 13.
    \end{multline*}
    Therefore, there exists an edge with at least $14$ quasi-incident vertices. 
\end{proof}

\begin{corollary} \label{corollary:compact:fat-faces}
Let $P$ be a compact right-angled hyperbolic $3$-polytope with $V > 80$ vertices. There is an edge $e \in P$ with  $k_i$-gonal faces around, $i=1, \ldots, 4$, such that  $\sum_{i=1}^4 k_i  \geqslant 24$. 
\end{corollary}

\subsection{Combinatorics of right-angled hyperbolic $3$-polyhedra with both finite and ideal vertices}

Let $P$ be a right-angled hyperbolic $3$-polyhedron with $V_F$ finite and $V_{\infty}$ ideal vertices. Let $E$ denote its number of edges, and $F$ be the number of its faces. The Euler characteristic of $P$ is $V_{F}+V_{\infty} - E + F = 2$. Since every ideal vertex is incident to four edges and each finite vertex is incident to three edges, we get $3V_{F}+4V_{\infty} = 2E$. Hence $F=\frac{1}{2} V_F+V_{\infty}+2$. 

We say that two faces are \emph{neighbours} if they have a common vertex.

\begin{lemma} \label{lemma:ideal-finite}
Let $P$ be a finite volume right-angled hyperbolic $3$-polyhedron with $V_F$ finite and $V_\infty$ ideal vertices. 
\begin{enumerate}
\item If $V_F + V_{\infty}>15$ and $V_{\infty} \geq 1$, then there is a face $f \in P$  with at least $6$ neighbours.
\item Let $V_F + V_\infty >17$ and $V_\infty \geq 3$. If there is no face with  $\geq 7$ neighbours, then there are at least $7$ faces such that each of them has $6$ neighbours.
\item If $V_\infty \geq 6$ and there is a face  $f \in P$ with at most $5$ neighbours, then there is a face $f' \in P$ with at least $7$ neighbours.
\end{enumerate}
\end{lemma}

\begin{proof}
Suppose that $P$ has $F$ faces. 
For a face $f_i \in P$, $i = 1, \ldots, F$, denote by $V_{F}^i$ the number of finite vertices in $f_i$ and by $V^i_{\infty}$ the number of ideal vertices in $f_i$.  Then the average number of neighbouring faces in $P$ is equal to
\begin{equation}
\frac{1}{F} \sum_{i} (2 V_{\infty}^i+ V_{F}^i) = \frac{1}{F}(8V_{\infty}+3V_{F}) = \frac{8V_{\infty}+3V_{F}}{V_{\infty}+\frac{1}{2}V_F+2}. \label{eqn100}
\end{equation} 

\underline{Part (1).} Our aim is to show that 
$$
\frac{8V_{\infty}+3V_{F}}{V_{\infty}+\frac{1}{2}V_F+2} > 5, 
$$
which is equivalent to 
$$
8 V_\infty + 3 V_F > 5 V_\infty + \frac{5}{2} V_F + 10 \
\Longleftrightarrow \
3 V_\infty + \frac{1}{2} V_F > 10.
$$
Using $V_F + V_{\infty} > 15$, we obtain 
$$
3 V_\infty + \frac{1}{2} V_F > 3 V_\infty + \frac{15}{2} - \frac{1}{2} V_\infty = \frac{5}{2} V_\infty + \frac{15}{2} \geq 10
$$
if $V_\infty \geq 1$. Thus, there is a face $f \in P$ with at least $6$ neigbouring faces.

\medskip

\underline{Part (2).} Let us use the formula~(\ref{eqn100}) for the average number $A$ of neighbours: 
$$
A = \frac{8V_{\infty}+3V_{F}}{V_{\infty}+\frac{1}{2}V_F+2} = 
\frac{6(V_{\infty}+\frac{1}{2}V_{F}+2)+2V_{\infty}-12}{V_{\infty}+\frac{1}{2}V_F+2} \geq 
\frac{6(V_{\infty}+\frac{1}{2}V_{F}+2)-6}{V_{\infty}+\frac{1}{2}V_F+2} = 6 -\frac{6}{V_{\infty}+\frac{1}{2}V_F+2}, 
$$
where we used $2 V_\infty - 6 \geq 0$. Since $V_\infty + V_F > 17$ and $V_\infty \geq 3$, we have  
$$
V_\infty + \frac{1}{2} V_F + 2 \geq 11 + \frac{1}{2} V_\infty \geq \frac{25}{2}.
$$ 
Hence the average number of neighbours satisfies the following inequality 
$A \geq 6 - \frac{12}{25} = \frac{138}{25}.$

Since $V_\infty + V_F > 17$ and $V_\infty \geq 3$, a polyhedron $P$ has $F=\frac{1}{2} V_F+V_{\infty}+2 \geq 11 + \frac{1}{2} V_\infty > 12$ faces. But $F$ is an integer number, so $F \geq 13$. Assume that $P$ has  $k \leq 6$ faces with $6$ neighbours. Hence $P$ has $F-k$ faces with at most $5$ neighbours, and an average $A$ of neighbours in $P$ satisfies the following inequality
$$
A \leq \frac{6 k + 5 (F - k)}{F} = 5 + \frac{k}{F} \leq 5 + \frac{6}{13} = \frac{71}{13}.
$$
Since $\frac{71}{13} < \frac{138}{25}$, we get a contradiction. Hence $P$ has at least $7$ faces such that each of them has $6$ neighbouring faces.  

\medskip

\underline{Part (3).} 
By the formula (\ref{eqn100}), using $V_\infty \geq 6$, we get the following inequality for the average number of neighbours:  
$$
A =  \frac{8V_{\infty}+3V_F}{V_{\infty}+\frac{1}{2}V_F+2} = \frac{6(V_{\infty}+\frac{1}{2}V_{F} + 2)+2V_{\infty}-12}{V_{\infty}+\frac{1}{2}V_F+2} \geq 6. 
$$
Since $A \geq 6$ and $f$ has at most $5$ neighbours, there is a face $f'$ with at least $7$ neighbours.
\end{proof}

\section{Proof of Theorem~\ref{theorem:main-ideal}} \label{sec3}

The Lobachevsky function is concave on the interval $[0, \frac{\pi}{2}]$ which implies that
$$\sum_{k=1}^m \Lambda\left(x_k\right) \leqslant m\,\Lambda\left(\frac{\sum_{k=1}^m x_k}{m}\right).
$$

Let $P$ be an ideal right-angled hyperbolic polyhedron and $v$ a vertex of $P$, which will further be called an \emph{apex}. For every face $f$ there is a unique projection $u$ of point $v$ to $f$. The projection will lie on the interior of $f$ unless $f$ meets one of the faces containing $v$. Projecting $u$ to the edges of $f$ will decompose $P$ into tetrahedra with three dihedral angles equal $\frac{\pi}{2}$, also known as \emph{orthoschemes}. Such a decomposition for the face formed by the vertices $v_1$, $v_2$, $v_3$, and $v_4$ is shown in Figure \ref{figure:ideal-face-triangulation}. Thus, we get eight tetrahedra having a common edge $vu$. Consider a tetrahedron formed by $v$, $u$, $w_4$, and $v_4$, where vertices  $v$ and $v_4$ are ideal, and vertices $u$ and $w_4$ are finite. Dihedral angles at edges $v w_4$, $w_4 u$ and $u v_4$ equal $\frac{\pi}{2}$. If dihedral angle at $vu$ equals $\alpha$, then dihedral angle at $v v_4$ equals $\frac{\pi}{2} - \alpha$ and dihedral angle at $v_4 w_4$ equals to $\alpha$. Thus, this tetrahedron s determined by $\alpha$, and we call $\alpha$ a \emph{parameter} of the tetrahedron. By \cite[Chapter~7]{Thurston},  volume of the tetrahedron formed by $v$, $u$, $w_4$, and $v_4$ equals $\frac{1}{2} \Lambda(\alpha)$.

\begin{figure}
    \centering
    
    \includegraphics[width=\textwidth/2-.5mm]{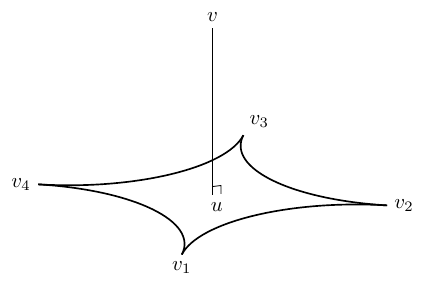}
    \includegraphics[width=\textwidth/2-.5mm]{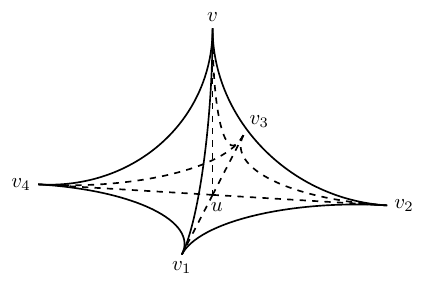}
    
    \includegraphics[width=\textwidth/2-.5mm]{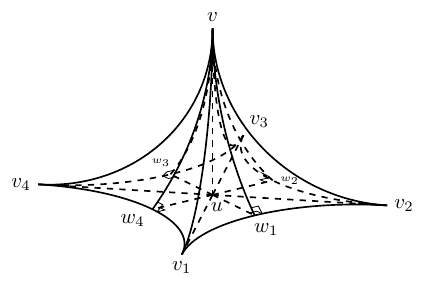}
    \includegraphics[width=\textwidth/2-.5mm]{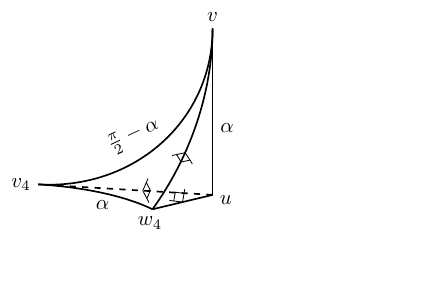}

    \caption{Decomposition of an ideal right-angled polyhedron.}
    \label{figure:ideal-face-triangulation}
\end{figure}

\begin{example}
Let $P$ be the antiprism $A(4)$ with the vertex set $V = \{ v_1, v_2, \ldots, v_8 \}$ .   
A decomposition of the $A(4)$ with apex at $v_1$ is shown in Figure \ref{figure:A4-triangulation}. Let us define a \emph{tetrahedral cone} $C(v)$ of the vertex $v$ as the union of tetrahedra of a decomposition containing $v$. Therefore, $A(4)$ splits in cones and 
$\Vol(A(4)) = \sum_{k = 2}^8 \Vol(C(v_k)).$
    
\begin{figure}[h]
    \centering

    \includegraphics[width=\textwidth/2-1mm]{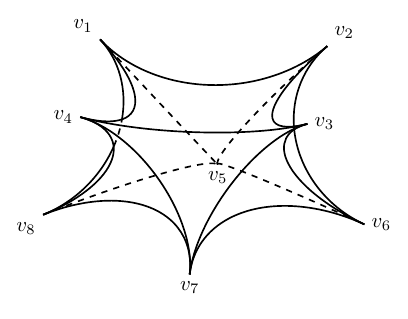}
    \includegraphics[width=\textwidth/2-1mm]{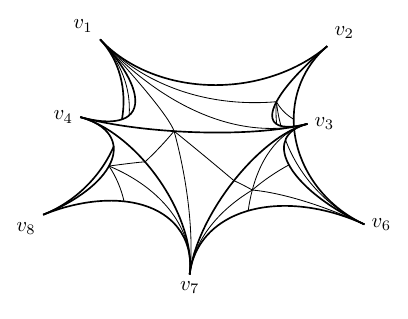}

    \caption{Antiprism $A(4)$ and its decomposition.}
    \label{figure:A4-triangulation}
\end{figure}

Vertices $v_2$, $v_4$, $v_5$, and $v_8$ are adjacent of $v_1$.    
The cone $C(v_4)$ consists of two tetrahedra with parameters $\alpha$ and $\beta$. Since the dihedral angle at the edge $v_1 v_4$ is $\frac{\pi}{2}$, then the sum $(\frac{\pi}{2} - \alpha) + (\frac{\pi}{2} - \beta)$ equals $\frac{\pi}{2}$ and $\alpha + \beta$ equals $\frac{\pi}{2}$. The concavity of $\Lambda$ implies
$$
\Vol(C(v_4)) = \frac{1}{2}\Lambda(\alpha) + \frac{1}{2}\Lambda(\beta) \leqslant \Lambda\left(\frac{\pi}{4}\right).
$$
The same holds for every vertex adjacent of $v_1$.  For more details see \cite[Proposition~5.2]{Atkinson}.
    
A similar argument provides $\Vol(C(v)) \leqslant 2\Lambda(\frac{\pi}{4})$ if $v$ is a quasi-adjacent of $v_1$ (i.e., $v_3$), and $\Vol(C(v)) \leqslant 4\Lambda(\frac{\pi}{4})$ for every other $v$ (i.e., $v_6$ and $v_7$). So
$$
\Vol(A(4)) \leqslant 4 \cdot \Lambda\left(\frac{\pi}{4}\right) + 1 \cdot 2 \Lambda\left(\frac{\pi}{4}\right) + 2 \cdot 4 \Lambda \left(\frac{\pi}{4} \right) = 14 \Lambda \left(\frac{\pi}{4}\right).
$$
\end{example}

\medskip

The previous example implicitly contains the proof of the following lemma.

\begin{lemma} \label{lemma:ideal-qn-volume-estimation}
    Let $P$ be an ideal right-angled hyperbolic polyhedron with $V$ vertices. If there is a vertex with $m$ quasi-adjacent vertices, then $\Vol(P) \leqslant \left(V - 4 - \frac{m}{2}\right) \cdot \frac{v_8}{2}.$
\end{lemma}

\begin{proof}
Every vertex has exactly four adjacent vertices. This fact provides
$$
\Vol(P) \leqslant (V - 1 - 4 - m) \cdot 4\Lambda \left(\frac{\pi}{4} \right) + m \cdot 2\Lambda \left(\frac{\pi}{4}\right) + 4 \cdot \Lambda \left(\frac{\pi}{4} \right) = \left(V - 4 - \frac{m}{2}\right) \cdot 4\Lambda \left(\frac{\pi}{4}\right).
$$
\end{proof}

\medskip

To prove item (3) of Theorem~\ref{theorem:main-ideal} we shall sum up the volumes of the tetrahedra not over the vertices but over the faces. A \emph{tetrahedral cone} $C(f)$ of the face $f$ is a union of the decomposition tetrahedra having a part of $f$ as a face. Such a cone is shown at the bottom left in Figure~\ref{figure:ideal-face-triangulation}.

\begin{lemma} \label{lemma:ideal-face-decomposition-estimation}
Let $P$ be a decomposed ideal right-angled hyperbolic polyhedron and $f$ a $k$-gonal face of $P$ that does not contain the apex. If $f$ is a quasi-incident to the apex, then
$$
\Vol(C(f)) \le (k - 1)\,\Lambda\Bigg(\frac{\pi}{2k - 2}\Bigg).
$$
If $f$ is not quasi-incident to the apex, then
$\Vol(C(f)) \le k\,\Lambda\left(\frac{\pi}{k}\right).$
\end{lemma}

\begin{proof}
    If $f$ is quasi-incident to the apex, then the projection of the apex does not lie inside $f$ and the cone $C(f)$ contains $2k - 2$ tetrahedra of the decomposition (like $v_3 v_4 v_7$ in Figure~\ref{figure:A4-triangulation}) and
 $$
 \Vol(C(f)) = \sum_{i = 2}^{2k - 1} \frac{1}{2} \Lambda(\alpha_i) \leqslant (k - 1)\,\Lambda\Bigg(\frac{\pi}{2k - 2}\Bigg), \quad  \text{ where } \sum_{i = 2}^{2k - 1} \alpha_i = \pi.
 $$

If $f$ is not quasi-incident to the apex, then the projection of the apex lies inside $f$ and the cone $C(f)$ contains $2k$ tetrahedra of the decomposition (like $v_3 v_6 v_7$ in Figure~\ref{figure:A4-triangulation}) and
 $$\Vol(C(f)) = \sum_{i = 1}^{2k} \frac{1}{2} \Lambda(\alpha_i) \leqslant k\,\Lambda\Bigg(\frac{\pi}{k}\Bigg), \quad \text{ where } \sum_{i = 1}^{2k} \alpha_i = 2\pi.
 $$
 The lemma is proved. 
\end{proof}

\medskip

Now let us prove Theorem~\ref{theorem:main-ideal}.

\begin{proof}
 Lemma~\ref{lemma:ideal-qn-volume-estimation}, Proposition~\ref{proposition:ideal-vertex-with-4-qn}, and Remark~\ref{remark:ideal-k-face-qn} provide items (1) and (2) of Theorem~\ref{theorem:main-ideal}.
    
    
Let $P$ be an ideal right-angled hyperbolic polyhedron with $V$ vertices and $F$ faces and without $k$-gonal faces for every $k \geqslant 5$. Then Proposition~\ref{proposition:ideal-isolated-vertex} and Lemma~\ref{lemma:ideal-face-decomposition-estimation} imply
 $$
 \Vol(P) \leqslant 8 \cdot 3 \Lambda\left(\frac{\pi}{3}\right) + 8 \cdot (4 - 1) \Lambda\left(\frac{\pi}{2 \cdot 4 - 2}\right) + (F - 20) \cdot 4\Lambda\left(\frac{\pi}{4}\right) = (V - 18) \cdot 4\Lambda\left(\frac{\pi}{4}\right) + 24 \cdot \left[\Lambda\left(\frac{\pi}{3}\right) + \Lambda\left(\frac{\pi}{6}\right)\right].
 $$ 
To complete the proof we recall that  $v_8 = 8 \Lambda \left( \frac{\pi}{4} \right)$ and $v_3 = 3 \Lambda \left( \frac{\pi}{3} \right) = 2 \Lambda \left( \frac{\pi}{6} \right)$.
\end{proof}

\section{Proof of Theorem~\ref{theorem:main-compact}} \label{sec4}

\subsection{Proof of part (1)}
Let us use the enumeration of faces as in~Figure~\ref{figure:edge-and-adjacent-faces}, with $f_1$ and $f_3$ containing $e$.  Since $V > 80$, then by Corollary~\ref{corollary:compact:fat-faces} there is an edge $e \in P$ such that for $k_i$-gonal faces $f_i$, $i=1, \ldots, 4$ we have  $\sum_{i=1}^4 k_i  \geq 24$. 

Let $P'$ be a polyhedron obtained by gluing $P$ with its image under reflection  in the plane passing through the face $f_1$.  Then $P'$ has $V' = 2V - 2k_1$ vertices. Denote by $f_2'$ a $(2k_2 - 4)$-gonal face of $P'$ containing $f_2$, by $f_3'$ a $(2k_3 - 4)$-gonal face of $P'$ containing $f_3$, and by $f_4'$ a $(2k_4 - 4)$-gonal face of $P'$ containing $f_4$. 

Recall the following volume bound from~\cite{EV20-2}.

\begin{lemma}[{\cite[Corollary~3.2]{EV20-2}}] \label{lemma:compact:faces-size-vol-estimation}
Let $P$ be a compact right-angled hyperbolic $3$-polyhed\-ron with $V$ vertices. Let $f_1$, $f_2$, and $f_3$ be three faces of $P$ such that $f_2$ is adjacent to both $f_1$ and $f_3$, and $f_i$ is $k_i$-gonal for $i = 1, 2, 3$. Then the following formula holds: 
$$ 
\Vol(P) \leqslant (V - k_1 - k_2 - k_3 + 4) \cdot \frac{5 v_3}{8}. \label{eqn2}
$$
\end{lemma}

Applying Lemma~\ref{lemma:compact:faces-size-vol-estimation} to the polyhedron $P'$ we get 
$$
2\Vol(P) = \Vol(P') \leqslant \left(2V - 2k_1 - (2k_2 - 4) - (2k_3 - 4) - (2k_4 - 4) + 4\right) \cdot \frac{5 v_3}{8}, 
$$
whence 
$$
\Vol(P) \leqslant (V - k_1 - k_2 - k_3 - k_4 + 8) \cdot \frac{5 v_3}{8} \leqslant (V - 16) \cdot \frac{5 v_3}{8} = \frac{5 v_3}{8} \cdot V - 10 v_3, 
$$
where we used   inequality  $\sum_{i=1}^4 k_i  \geq 24$. 

\subsection{Proof of part (2)}
Denote $k_1 = k$. Since any face of $P$ has at least 5 sides, we have $k_i \geq 5$ for $i=2,3,4$. Then, from preceding inequality we get 
$$
\Vol(P) \leqslant (V - k - 5 - 5 - 5 + 8) \cdot \frac{5 v_3}{8} = \frac{5 v_3}{8} \cdot V - \frac{5k + 35}{8} v_3.
$$ 
Thus, Theorem~\ref{theorem:main-compact} is proved. 


\section{Proof of Theorem~\ref{theorem:main-compact-ideal}} \label{sec5}

Observe that formula (\ref{eqn1}) implies the following upper bound holds for a polyhedron $P$ with $V_\infty$ ideal and $V_F$ finite vertices: 
$\Vol(P) <  \frac{v_8}{2} \cdot V_\infty + \frac{5v_3}{8} \cdot V_F - \frac{v_8}{2}.$

\medskip 

Let $P^1$ denote the polyhedron $P$, and let $k_\infty^1$, resp. $k_F^1$, be the number of ideal, resp. finite, vertices of a face $f^1$ of $P^1$. Consider the union $P^2$ of two copies of $P^1$ glued along the face $f^1$. Then $P^2$ is right--angled with $V^2_\infty = 2V_\infty - k^1_\infty$ ideal vertices and $V^2_F = 2V_F - 2k^1_F$ finite vertices. Then applying to $P^2$ the upper bound from (\ref{eqn1}), we get 
$$
\Vol(P) = \frac{\Vol(P^2)}{2} < \frac{v_8}{2} \cdot V_\infty + \frac{5v_3}{8} \cdot V_F - \left(  \frac{v_8}{4} \cdot k^1_\infty  + \frac{5v_3}{8} \cdot k^1_F + 
\frac{v_8}{4} \right) =  \frac{v_8}{2} \cdot V_\infty +  \frac{5v_3}{8} \cdot V_F - c_1 - \frac{v_8}{4}, 
$$ 
where 
$c_1 =  \frac{v_8}{4} \cdot k^1_\infty  +  \frac{5v_3}{8} \cdot k^1_F$.

\medskip

Let $k_\infty^2$, resp. $k_F^2$, be the number of ideal, resp. finite, vertices of a face $f^2$ of $P^2$. Consider the union $P^3$ of two copies of $P^2$ glued along the face $f^2$. Then $P^3$ is right-angled with 
$V^3_\infty = 4V_\infty -2 k^1_\infty - k^2_\infty$
ideal vertices and 
$V^3_F = 4V_F - 4k^1_F - 2k^2_F$
finite vertices. Applying to $P^3$ the upper bound from (\ref{eqn1}), we get
$$
\Vol(P) = \frac{\Vol(P^3)}{4} < \frac{v_8}{2} \cdot V_\infty + \frac{5v_3}{8} \cdot V_F - c_1 -  c_2 - \frac{v_8}{8}, 
$$ 
where 
$c_2 = \frac{v_8}{8} \cdot k^2_\infty + \frac{5v_3}{16} \cdot k^2_F$.

\medskip

Continuing the process inductively, we obtain
\begin{equation} 
\Vol(P) < \frac{v_8}{2} \cdot V_\infty + \frac{5v_3}{8} \cdot V_F - c_1 -  c_2 - \ldots - c_n - \frac{v_8}{2^{n+1}},  \label{eqn200}
\end{equation} 
where for $i=1, \ldots, n$ we have 
$c_i = \frac{v_8}{2^{i+1}} \cdot  k^i_\infty + \frac{5v_3}{2^{i+2}} \cdot k^i_F$.

\medskip

Suppose that for every $i = 1, 2, \ldots,$ the face $f^i$ is \textit{not} an ideal triangle and has at least $6$ neighbouring faces (as Lemma~\ref{lemma5.2} shows later on, we can choose a face $f^i$ of $P^i$ with this property for any $i$). Then the value $c_i$ is minimal for $k^i_\infty = 2$ and $k^i_F = 2$. In this case we compute
$$
\sum_{i=1}^n c_i = \left(v_8 +  \frac{5v_3}{2}\right) \cdot \left(\sum_{i=1}^n \frac{1}{2^i}\right) = \left(v_8 +  \frac{5v_3}{2}\right) \cdot \left(1 - \frac{1}{2^n}\right). 
$$
Taking the limit $n \to \infty$ in (\ref{eqn200}), we obtain 
$$
\Vol(P) \leq \frac{v_8}{2} \cdot V_\infty + \frac{5v_3}{8} \cdot V_F - \left( v_8 + \frac{5v_3}{2} \right).
$$ 

\medskip

Let $N_6(P)$ denote the set of faces of a polyhedron $P$ with the following property: $f \in N_6(P)$ if $f$ has at least $6$ neighbouring faces. This set is non-empty by part (1) of Lemma~\ref{lemma:ideal-finite}.

\begin{lemma} \label{lemma5.2}
If $V_\infty + V_F > 17$ then for any $i=1, \ldots, n$ the set $N_6\left(P^i\right)$ contains at least one face that is not an ideal triangle.
\end{lemma}

\begin{proof} 
Let us observe that for any $i \geq 1$ the polyhedron $P^i$ is not an octahedron. Indeed, this holds true for $P^1 = P$ since  $V^1_{\infty} = V_{\infty}$, $V^1_F = V_F$ and $V_\infty^1 + V_F^1 > 17$. 
Assume that $P^2$ is an octahedron. Then $2 V_\infty^1 \ge V_\infty^2 = 6$ and 
$$
\Vol\left(P^2\right) = 2 \cdot \Vol\left(P^1\right) \ge 2 \cdot \frac{4 V_{\infty}^1 + V_F^1 - 8}{32} \cdot v_8 >  \frac{4 V_{\infty}^1 + 17 - V_{\infty}^1 - 8}{16} \cdot v_8 \geq \frac{18 \cdot v_8}{16} > v_8,
$$
which is a contradiction. Finally, let us show that for any $i \geq 3$ polyhedron $P^i$ is not an octahedron either. If $P^i$ was an octahedron, then it would hold that $V_\infty^1 \ge 1$ and by inequality~(\ref{eqn1}) we would obtain
$$
\Vol\left(P^i\right) \ge 2^{i - 1} \cdot \frac{4 V_{\infty}^1 + V_F^1 - 8}{32} \cdot v_8 > 2^{i - 6} \cdot (4 V_{\infty}^1 + 17 - V_{\infty}^1 - 8) \cdot v_8 \geq 12 \cdot 2^{i - 6} \cdot v_8  = \frac{3}{2} \cdot 2^{i-3} \cdot v_8 > v_8,
$$
which is again a contradiction.

\medskip 

Let $i = 1$ and assume that all faces from  $N_6(P^1)$ are ideal triangles. Then part (2) of Lemma~\ref{lemma:ideal-finite} implies that $N_6(P^1)$ contains at least $7$ ideal triangles. Denote the set of the faces of $P^1$ by $\mathcal{F}$ and the number of the ideal vertices of a face $f$ by $I(f)$. Suppose that $P^1$ contains at most $5$ ideal vertices. Then
$$
21 = 3 \cdot 7 \le \sum_{f \in \mathcal{F}} I(f) \le 4 \cdot 5 = 20.
$$
Thus, $P^1$ contains at least $6$ ideal vertices.

Since  $P^1$ is not an octahedron, then there is a face $f'$ that is not an ideal triangle. Therefore, $f' \not\in N_6(P^1)$, whence $f'$ has at most $5$ neighbouring faces. Then by part (3) of Lemma~\ref{lemma:ideal-finite} there is a face $f''$ which has at least $7$ neighbouring faces. Therefore, $f''$ is not an ideal triangle. However $f'' \in N_6(P^1)$, that contradicts to the assumption.  

\medskip
Now let $i \geq 2$ and assume that for some $i \ge 2$ each face from  $N_6\left(P^i\right)$ is an ideal triangle.
The polyhedron $P^{i}$ is the union of two copies of $P^{i-1}$ along the face $f^{i-1}$. Let $D^{i-1}$ be the set of faces of $P^{i-1}$ that have a common edge with $f^{i-1}$. Let $S^{i}$ denote the set of such faces of $P^{i}$ that contain a face from $D^{i-1}$. That is, $S^{i}$ consists of all of the new faces that appeared after the union of two copies of $P^{i-1}$ along $f^{i-1}$. By Theorem~\ref{theorem:Andreev}, each face of a right-angled polyhedron has at least $5$ neighbours. Hence each face from the set $S^{i}$ has at least $6$ neighbours and $S^i \subset N_6(P^i)$. Therefore, by our assumption, each face from $S^i$ is an ideal triangle. Then each face from $D^{i - 1}$ is a triangle with two ideal and one finite vertices. Moreover,  $f^{i - 1}$ is a face with an even number of vertices, such that its ideal and finite vertices alternate among themselves. Namely, if $f^{i-1}$ has $2k$ vertices, then there are $k$ ideal and $k$ finite vertices, and $k \geq 2$. 

Observe that there are at least $2$ ideal vertices in $P^{i - 1}$ that are not contained in the face $f^{i - 1}$. Indeed, since all faces from $D^{i-1}$ are triangles with $2$ ideal vertices, then $P^{i-1}$ has at least one ideal vertex $v$ that is not contained in the face $f^{i-1}$. Suppose that $v$ is the only such vertex. Then $v$ is incident to all of the vertices of $f^{i-1}$. If $k \geq 3$, then $v$ is a vertex of valency $2k$. The latter is impossible by Theorem~\ref{theorem:Andreev}. If $k = 2$ then $P^{i-1}$ is a quadrilateral pyramid with $5$ faces, while by Theorem~\ref{theorem:Andreev} every right-angled polyhedron has at least $6$ faces. Thus, $P^{i-1}$ should have at least $2$ ideal vertices that are not contained in the face $f^{i-1}$. 

Hence $P^i$ has at least $k + 4 \geq 6$ ideal vertices. Since $P^i$ is not an octahedron, then it has at least one face $f'$ that is not an ideal triangle. Using part (3) of Lemma~\ref{lemma:ideal-finite} we obtain that there is a face $f''$ of $P^i$ that has at least $7$ neighbours and thus cannot be an ideal triangle. This contradicts our assumption about $N_6\left(P^i\right)$, and the lemma is proved. 
\end{proof}

\end{document}